\documentclass[11pt]{article}
\usepackage[utf8]{inputenc}
\usepackage{geometry}
\usepackage{aligned-overset}
\usepackage{amsmath}
\usepackage{amsfonts}
\usepackage{amssymb}
\usepackage{amsthm}
\usepackage{graphicx}
\usepackage{enumitem,mdwlist}
\setlist[itemize]{topsep=0ex,itemsep=0ex,parsep=0.4ex}
\setlist[enumerate]{topsep=0ex,itemsep=0.1ex,parsep=0.4ex}
\usepackage{bm}
\usepackage{caption}
\usepackage{subcaption}
\usepackage{url}
\usepackage[linktoc = all, hidelinks, colorlinks, unicode=true]{hyperref} 
\hypersetup{linkcolor={blue}, citecolor={green!60!black}, urlcolor={black}}
\usepackage[table,xcdraw]{xcolor}
\usepackage{mathtools}
\usepackage{systeme}
\usepackage[capitalise]{cleveref}
\usepackage{thm-restate}
\usepackage{todonotes}
\usepackage{bbm}
\usepackage{comment}
\usepackage{dsfont}
\usetikzlibrary{shapes,arrows,calc}
\usepackage{setspace}
\usepackage[square,sort,comma,numbers]{natbib} 
\newtheorem{theorem}{Theorem}[section]
\newtheorem{proposition}[theorem]{Proposition}

\newtheorem{conjecture}[theorem]{Conjecture}
\newtheorem{lemma}[theorem]{Lemma}
\newtheorem{corollary}[theorem]{Corollary}
\newtheorem{fact}[theorem]{Fact}

\newtheorem*{remarks*}{Remarks}
\newtheorem*{remark*}{Remark}

\crefname{equation}{}{}
\crefname{item}{}{}

\numberwithin{equation}{section}

\newcommand{\inv}{^{-1}}
\renewcommand{\geq}{\geqslant}
\renewcommand{\leq}{\leqslant}

\Crefname{conjecture}{Conjecture}{Conjectures}
\Crefname{claim}{Claim}{Claims}
\Crefname{lemma}{Lemma}{Lemmas}
\Crefname{fact}{Fact}{Facts}
\Crefname{subsection}{Subsection}{Subsections}
\Crefname{figure}{Figure}{Figures}

\newcommand{\ceil}[1]{{\left\lceil #1 \right\rceil}}
\newcommand{\defi}[1]{\emph{\color{red!50!black}#1}}

\newcommand{\ella}[1]{{\color{purple}{Ella: #1}}}

\def\NN{\mathbb N}
\def\inv{^{-1}}

\def\Free{\mathrm{Free}}

\let\eps=\varepsilon
\let\theta=\vartheta
\let\phi=\varphi

\newcommand{\cT}{\mathcal{T}}

\theoremstyle{definition}

\def\namedlabel#1#2{\begingroup
    #2%
    \def\@currentlabel{#2}%
    \phantomsection\label{#1}\endgroup
}

\def\ew#1{}
	\renewcommand{\ew}[1]{\footnote{\textbf{EW:}#1}}
    
\def\ella#1{}
	\renewcommand{\ella}[1]{\textcolor{orange}{\textbf{Ella:}#1}}

\def\eps{{\varepsilon}}
\title{Are trees really just butterflies in disguise?
}

\author{
Giovanne Santos\footnotemark[1]
\and
Maya Stein\footnotemark[1]
\and 
Ella Williams\footnotemark[2]
}

\date{\today}

\begin{document}
\maketitle

\begin{abstract}
    As a generalisation of the Erd\H{o}s-S\'os conjecture about graphs, Addario-Berry, Havet, Linhares Sales, Reed and Thomassé conjectured that every digraph on $n$ vertices with more than $(k-1)n$ arcs contains every antidirected tree with $k$ arcs. We prove a dense, approximate version of this for trees with bounded maximum degree, as well as for trees whose layers are evenly distributed. We use a regularity based approach, centred around finding a copy of a given tree in the blow up of a caterpillar.
\end{abstract}

\renewcommand{\thefootnote}{\fnsymbol{footnote}}
\footnotetext[1]{Departamento de Ingenier\'ia Matem\'atica y Centro de Modelamiento
Matem\'atico (CNRS IRL2807), Universidad de Chile, Santiago, Chile. Email: {\texttt{\{gsantos,mstein\}@dim.uchile.cl}}. Supported  by ANID grant CMM Basal FB210005. MS also supported by ANID Fondecyt Regular Grant 1260024.}
\footnotetext[2]{Department of Mathematics, University College London, UK. Email: \href{mailto:ella.williams.23@ucl.ac.uk}{\texttt{ella.williams.23@ucl.ac.uk}}. Supported by the Martingale Foundation.}

\section{Introduction}

The Erd\H{o}s-S\'os conjecture~\cite{Erdos63}, posed in 1964, is one of the central problems in extremal graph theory concerning the appearance of trees. It states that every graph $G$ on $n$ vertices with more than $(k-1)n/2$ edges contains a copy of every tree with $k$ edges.
Besomi, Pavez-Sign\'e and Stein~\cite{besomi2021erdHos}, and independently Rohzoň~\cite{rohzon19} proved this conjecture approximately, if the host graph is dense (meaning the size of the embedded tree is linear in the size of the graph), and the tree has sublinear maximum degree. For trees with maximum degree $O(1)$ and $k$ large, Pokrovskiy~\cite{AP_hyperstability} recently proved that the conjecture is true.

A natural question is whether an analogous statement holds for directed graphs. Unlike the undirected setting, however, not every orientation of a tree can be forced by an average degree condition of this type alone. The obstruction comes from a complete bipartite orientation, which contain no directed path of length two. Consequently, the only class of directed trees for which such a statement can hold is the class of antidirected trees, that is, those in which every vertex has either out-degree $0$, or in-degree $0$. 

Indeed, let $D$ be the complete bipartite digraph with bipartition $(A,B)$ and all arcs directed from $A$ to $B$. Then $D$ contains no directed path of length two, and hence no non-antidirected tree. Since $e(D)= |A||B|$, choosing $|A|$ and $|B|$ sufficiently large yields digraphs with arbitrarily large average degree avoiding every non-antidirected tree.

For antidirected trees, an analogue of the Erd\H{o}s-S\'os conjecture has been posed.
 
\begin{conjecture}[Addario-Berry, Havet, Linhares Sales, Reed and Thomassé \cite{Addarioberry2013oriented}]\label{conj:main}
 Every digraph on $n$ vertices with more than $(k-1)n$ arcs contains a copy of every antidirected tree with $k$ arcs.
\end{conjecture}

This conjecture implies the Erd\H{o}s-S\'os conjecture, as for any graph $G$, we can consider a corresponding digraph with $2e(G)$ arcs, constructed by taking both arcs for each edge of $G$.

The authors of~\cite{Addarioberry2013oriented} proved that Conjecture~\ref{conj:main} holds for trees with diameter at most three.
Stein and Trujillo-Negrete~\cite{stein_trujillo-negrete2025} proved that it is true for large digraphs not containing any of three forbidden subgraphs, as well as proving it holds for antidirected caterpillars~\cite{stein_trujillo-negrete2025}. Stein and Zárate-Guerén~\cite{Stein_Zarate-Gueren_2024} proved an approximate version of Conjecture~\ref{conj:main} for balanced antidirected trees of maximum degree $(\log n)^{o(1)}$, when the host graph is dense and oriented. Our main result improves this for trees with bounded maximum degree, by removing the balanced condition and not requiring the host graph to be oriented.  
\begin{theorem}\label{thm:main_dense}
    For all $\alpha>0$ and $\Delta\in \NN$ there exists $k_0$ such that the following holds for all $n\geq k>k_0$ with $k\geq \alpha n$. Every digraph  $D$ on $n$ vertices with $e(D)\geq (1+\alpha)kn$ contains a copy of every antidirected tree $T$ on $k$ arcs with $\Delta(T)\leq \Delta$.
\end{theorem}

As applications of Theorem~\ref{thm:main_dense}, presented in \cref{sec:applications}, we obtain new bounds on Burr's conjecture for bounded degree antidirected trees and dense host graphs, as well as new upper bounds on the corresponding directed and oriented Ramsey numbers.

Alongside \cref{thm:main_dense}, one may ask for a similar result for trees that do not have constant maximum degree. We focus on trees that are `well distributed', in the sense that, we can choose a root of the tree, such that there are not too many vertices having the same distance to the root. We think of these trees as having roughly balanced layers.

\begin{theorem}\label{thm:balance_trees}
    For all $\alpha>0$ there exist $\gamma \in (0,1)$ and $k_0\in \NN$, such that for all $n\geq k> k_0$ satisfying $k\geq \alpha n$, the following holds.  Let $T$ be a rooted antidirected tree on $k$ arcs, such that there are at most $\gamma k$ vertices in $T$ that are at distance exactly $i$ from the root, for all $i\geq 0$.
    Every digraph  $D$ on $n$ vertices with $e(D)\geq (1+\alpha)kn$ contains a copy of $T$.
\end{theorem}

We remark that the two classes of trees given in \cref{thm:main_dense} and \cref{thm:balance_trees} are genuinely distinct. For example, an antidirected $\Delta$-ary tree with root $r\in V(T)$ clearly belongs to the class of antidirected trees with maximum degree at most $\Delta$. However, there are $\Delta^i$ vertices in $T$ of distance $i$ from $r$, and if $e(T) = k$ for large $k$, the number of vertices at furthest distance from $r$ is much larger than $\gamma k$. The same is true if we choose any other vertex of $T$ to be the root, and thus this tree is not covered by \cref{thm:balance_trees}. On the other hand, the tree obtained by taking a star with $\gamma k$ leaves, all directed outwards, and subdividing each of the arc $\gamma^{-1}$ times, maintaining antidirectedness, has maximum degree $\omega(1)$ and bounded layer size, seen by choosing the root to be the centre of the original star. Many similar examples can be constructed.

\subsection{Discussion of recent AI breakthrough on the Erd\H{o}s--S\'{o}s conjecture}
Whilst in the final stages of proofreading this manuscript, it was announced that GPT-6 Astra proved the Erd\H{o}s--S\'{o}s conjecture in full, using a counting argument.
We note that their proof also extends to proving \cref{conj:main} and give details of this generalisation in \cref{sec:AI}.
A Lean~4 formalisation of this generalisation is available at~\cite{DLean}.

As this work was done in parallel, we have not changed the paper to reflect this recent update, other than adding \cref{sec:AI}. In particular, the bounds obtained in \cref{sec:applications} as applications of \cref{thm:main_dense} can be generalised as well.

\textbf{Statement of AI use:}
AI was not used in the production of this manuscript. The only exception is the proof in~\cref{sec:AI}, which adapts
the AI generated proof of Erd\H{o}s--S\'{o}s conjecture. The adaptation was written entirely by the authors,
without any further use of AI. All other ideas in this paper were thought of and written entirely by the authors.

\section{Preliminaries}

\subsection{Notation}
Let $G$ be a graph. We write $|G| = |V(G)|$ and $e(G) = |E(G)|$ for the number of vertices and number of edges respectively. The minimum degree and maximum degree of $G$ are denoted by $\delta(G)$ and $\Delta(G)$. For $U\subseteq V(G)$, we write $N_G(U)$ for the union of all neighbours of vertices in $U$.

Let $D$ be a digraph with vertex set $V(D)$ and arc set $E(D)$, and let $x\in V(D)$. We write $|D| = |V(D)|$ and $e(D) = |E(D)|$.
The set of all in-neighbours and the set of all out-neighbours of $x$ in $D$ are denoted by $N_D^-(x)$ and $N_D^+(x)$ respectively, and their sizes are denoted by the in-degree $\deg_D^-(x)$ and out-degree $\deg_D^+(x)$ respectively. The minimum in-degree and maximum in-degree of $D$ are denoted by $\delta^-(D)$ and $\Delta^-(D)$, and analogously for out-degree with $\delta^+(D)$ and $\Delta^+(D)$. The maximum total degree of $D$, denoted by $\Delta(D)$, is the minimum $m\in \NN$ such that every vertex is contained in at least $m$ arcs.

We write $N_D(x) = N_D^-(x)\cup N_D^+(x)$ for the total neighbourhood of $x$ and $\deg_D(x) = |N_D(x)|$ for the total degree of $x$. 
Let $X\subseteq V(D)$. We write $N^+_D(X)$ for the union of all out-neighbours of vertices in $X$, and $N^-_D(X)$ and $N_D(X)$ are defined analogously.
We may omit subscripts where context is clear.
For convenience, we use `in' and `$-$' interchangeably, and do the same for `out' and `$+$'.

A digraph $D$ is \defi{oriented} if there is at most one arc between every pair of vertices. $D$ is \defi{antidirected}
if it is oriented and contains no directed path of length two. Note that this is equivalent to all vertices of $D$ having in-degree $0$ or out-degree $0$. We say that a vertex $v\in V(D)$ is a \defi{source}, or an \defi{out-vertex}, if it has no in-neighbours, and similarly it is a \defi{sink}, or an \defi{in-vertex} if it has no out-neighbours.

Given a rooted tree $T$ (either directed or undirected), we denote the root of $T$ by $r(T)$.

A \defi{caterpillar} is a tree that consists of a path, which we usually call the \defi{spine}, with leaves attached to it.

\subsection{Short sketch}

In order to prove \cref{thm:main_dense,thm:balance_trees}, we use the diregularity lemma (see \cref{sec:diregularity}), to obtain a reduced graph $R$ of a similar density to $D$, in particular $e(R)\geq (1+\alpha')k|R|$ for some $\alpha'\in (0,\alpha)$. As mentioned earlier, Stein and Trujillo-Negrete proved that \cref{conj:main} holds for $k$-arc antidirected caterpillars.

\begin{proposition}[\cite{stein_trujillo-negrete2025}]\label{prop:ST_caterpillar}
	Every digraph $D$ on $n$ vertices with more than $(k-1)n$ arcs contains every antidirected caterpillar 
	with $k$ arcs. 
\end{proposition}

We will apply this result to $R$. Roughly speaking, it remains to show that, for every antidirected tree~$T$ satisfying either of the descriptions in \cref{thm:main_dense} or \cref{thm:balance_trees}, there exists a caterpillar $C$ on roughly~$k|R|/|D|$ vertices such that there is a homomorphism from $T$ to $C$ in which not too many vertices of $T$ are mapped to the same vertex in $C$. We will then use the properties of the regular pairs in $R$ to find the embedding of $T$ in $D$. In order to find this homomorphism,  we need to partition $T$ into small pieces that will be used to determine the construction of the caterpillar. For trees with bounded layer size, as in \cref{thm:balance_trees}, we use the decomposition into layers. For bounded degree trees, we use the following standard tree-separating lemma (see \cite{BPS1}).

\begin{lemma}[Proposition 4.1 in \cite{BPS1}]
  \label{lemma:cutting_1}
  For every $\beta > 0$, and for all rooted trees $T$ on $k$ edges,
  there is a set $S \subseteq V(T)$ and a family $\mathcal{T}$ of disjoint
  rooted trees such that 
  \begin{enumerate}[label = \upshape{(\alph*)}]
    \item $r(T) \in S$;
    \item $\mathcal{T}$ consists of the connected components of $T - S$, and every $T'\in \mathcal{T}$
          is rooted at the vertex closest to $r(T)$;
    \item $|T'| \leq \beta k$, for every $T' \in \mathcal{T}$; and
    \item $|S| \leq \frac{1}{\beta} + 2$.
  \end{enumerate}
\end{lemma}
  We say that such a pair $(S,\cT)$ is a \defi{$\beta$-decomposition} of $T$. 

\subsection{Diregularity}\label{sec:diregularity}
Let~$D$ be a digraph, and let $X,Y \subseteq V(D)$ be nonempty and disjoint. We let
\[
  \text{$d^{+}(X,Y) := \frac{|(X\times Y)\cap E(D)|}{|X||Y|}$ \ and \ $d^{-}(X,Y):=d^+(Y,X)$.}
\]
 For~$\eps > 0$ we say~${X' \subseteq X}$
is~\defi{$\eps$\nobreakdash-significant} if~${|X'| \geq \eps|X|}$.
For~${\diamond \in \{+,-\}}$, the pair~$(X,Y)$
is~\defi{$(\eps,\diamond)$-regular}
if~${|d^{\diamond}(X,Y) - d^{\diamond}(X',Y')| \leq \eps}$ for
all~{$\eps$\nobreakdash-significant} subsets~${X' \subseteq X}$
and~${Y' \subseteq Y}$. If furthermore ${d^{\diamond}(X,Y) \geq d}$ for
some~${d\geq 0}$, we call~$(X,Y)$ \defi{$(\eps,\diamond,d)$-regular}. Given an $(\eps,\diamond,d)$-regular pair $(X, Y)$ and an $\eps$-significant $Y'\subseteq Y$, a vertex $v\in X$ is called $\diamond$-\defi{typical} to $Y'$ if $d^\diamond(\{v\},Y')\geq d-\eps$. It is well-known that between a regular pair, almost all vertices are typical to any significant subset of the opposite side.

\begin{fact} [see \cite{komlos2000regularity}] \label{fact:typical} 
    Let $(X, Y)$ be an $(\eps,\diamond,d)$-regular pair, and let $\eps\leq \delta \leq 1/2$. Then the following holds:
    \begin{enumerate}[label = \upshape{(\alph*)}]
        \item For each $\eps$-significant set $Y'\subseteq Y$, all but at most $\eps|X|$ vertices of $X$ are $\diamond$-typical to $Y'$.
        \item For all $\delta$-significant sets $X' \subseteq X$ and $Y' \subseteq Y$, the pair $(X',Y')$ is $(\frac{\eps}{\delta}, \diamond, d-\eps)$-regular.
    \end{enumerate}
\end{fact}

Szemerédi’s regularity lemma~\cite{szemeredi} states that every large graph can be partitioned into a bounded
number of vertex sets, most of which are pairwise $\eps$-regular. We will use the following analogue of Szemerédi’s result to digraphs, obtained by Alon and Shapira~\cite{alon2003testing}.

\begin{lemma}[Degree form of the diregularity lemma \cite{alon2003testing}]
  \label{lemma:diregularity}
  For every $0<\eps < 1$ and $m_0 \in \NN$, there are integers $M_0$ and $n_0$
  such that the following holds for all $n \geq n_0$ and $d \leq 1$.
  If $D$ is a digraph on $n$ vertices, then there is a partition of $V(D)$
  into $V_0,V_1,\dots,V_t$ and a spanning subgraph~$D'$ of~$D$ such that:
  \begin{enumerate}[label = \upshape{(\alph*)}]
    \item $m_0 \leq t \leq M_0$,
    \item $|V_0| \leq \eps n$ and $|V_1|=\dots=|V_t|$,
    \item $\deg_{D'}^{\diamond}(v) > \deg_{D}^{\diamond}(v) - (d+\eps)n$,
      for all $v \in V(D)$ and $\diamond \in \{+,-\}$,
    \item $V_i$ is an independent set in $D'$, for all $i \in [t]$, and
    \item for all $1 \leq i,j \leq t$, the ordered pair $(V_i,V_j)$ is
      $(\eps,+)$-regular in $D'$ with density either $0$ or at least~$d$.
  \end{enumerate}
\end{lemma}

Let $D$ be a digraph on $n$ vertices, let~${\eps, d < 1}$, and let $V_0,\dots,V_r$ be a partition
of $V(D)$ as in Lemma~\ref{lemma:diregularity}. An \defi{$(\eps,d)$-reduced digraph} $R$
of $D$ with respect to $V_0,\dots,V_t$ is the digraph with vertex set~${V(R) := \{V_1,\dots,V_t\}}$, and arc set $E(R):=\{(V_i, V_j)|\text{ the pair~$(V_i,V_j)$ is~$(\eps,+,d)$-regular}\}$. We refer to the vertices of $R$ as \defi{clusters}. It is commonly known that the reduced graph inherits several properties from $D$, arc density being one. 

\begin{fact}
    \label{fact:reg_edges}
    Let $0<3\eps\leq d\leq \rho/2$. If~$D$ is a digraph on~$n$ vertices with at least $\rho n^2$ arcs and~$R$ is
  an~$(\eps, d)$-reduced digraph of~$D$,
  then $R$ has at least $(\rho - 2d)|R|^2$ arcs. 
\end{fact}

\section{Proof of Theorem~\ref{thm:main_dense}}

\textbf{Step 0 (Setting up)}.
Given $\alpha > 0$ and $\Delta\in \NN$, we introduce new constants $k_0, \beta, \eps$ and $d$ such that
\[  
1/k_0 \ll \beta \ll \eps \ll d \ll \alpha,\Delta\inv.
\]

Let $n \geq k> k_0$, and let $D$ be a digraph on $n$ vertices with at least $(1+\alpha)kn$ arcs
where $k\geq \alpha n$, and let~$T$ be an antidirected tree on $k$ arcs with~${\Delta(T)\leq \Delta}$.

\textbf{Step 1 (Preparing the host digraph)}. Apply the diregularity lemma (\cref{lemma:diregularity}) to $D$ with $\eps$ and~${m_0 = 1/\eps}$
to get a subgraph $D'$, and a reduced digraph $R$ on vertex set $\{V_1,\dots,V_t\}$. By \cref{fact:reg_edges} with~${\rho = (1+\alpha)k/n}$,
we have~${e(R)\geq (1+\alpha^2)k|R|^2/n}$. For convenience, let~${\kappa = k|R|/n}$ and $m$ be the size of the clusters in~$R$.

\textbf{Step 2 (Preparing the tree)}. Root $T$ at an arbitrarily chosen vertex $r(T) \in V(T)$.
Apply \cref{lemma:cutting_1} with $\beta$ to the underlying graph of $T$ to get a~$\beta$-decomposition~$(S,\cT)$, where $S\subseteq V(T)$ has size at most~$2/\beta$ and $\mathcal{T}$ consists of the connected components of $T-S$, each of which is rooted and has size at most $\beta k$. We refer to the vertices in $S$ as \defi{seeds}, and to the rooted subtrees $T'\in \mathcal{T}$ as the \defi{pieces} of the decomposition. The parent of a seed vertex $s\in S$ is the parent of $s$ in $T$, and the parent of a piece $T'\in \mathcal{T}$ is the unique vertex that is the parent of $r(T')$ in $T$. Note that since $\Delta(T)\leq \Delta$, and the root of each piece is the child of a seed, we have $|\mathcal{T}|\leq \Delta |S| \leq 2\Delta/\beta$.

\textbf{Step 3 (Finding a good assignment)}. In this step we find a homomorphism $\phi:V(T) \longrightarrow V(R)$ such that
\begin{equation}\label{eq:assignment}
    |\phi\inv(V_i)| < (1- \eps^{1/4})|V_i|\;\text{ for all } i \in [t].
\end{equation}

In order to find such homomorphism, we first find a rooted antidirected caterpillar $C$ and
a homomorphism~$\psi:V(T)\longrightarrow V(C)$ satisfying the following property for
all $v \in V(C)$, except for two vertices,
\begin{equation}
    \label{eq:full_cat}
    (1-\alpha^2/8)m \leq |\psi\inv(v)|\leq (1-\eps^{1/4})m.
\end{equation}

Note that this property, together with $m\geq (1-\eps)n/|R|$, implies
\begin{align*}
    k +1 = \sum_{v\in V(C)}|\psi\inv(v)| \geq (|C|-2)(1-\alpha^2/8)m\geq (|C|-2)(1-\alpha^2/4)\frac{k}{\kappa},
\end{align*}
which since $\kappa \geq \alpha|R|\geq \alpha m_0\geq \alpha/\eps \geq 4/\alpha^2$, gives 
\begin{equation}\label{eq:size_of_cat}
    |C|\leq \frac{\kappa}{1-\alpha^2/4}+3 \leq (1+\alpha^2/2)\kappa+2 \leq (1+\alpha^2)\kappa.
\end{equation}

We let $s := |S \cup \cT|$, and note that $s\leq 4\Delta/\beta$.
We construct $C$ and $\psi$ in $s$ steps. Let $(\pi_0,\dots,\pi_s)$ be an ordering of~${S \cup \cT}$ such
that $\pi_0 = r(T)$ and $T[\cup_{j \leq i} \pi_j]$ is connected, for all $i \leq s$. We define $T_i := T[\cup_{j \leq i} \pi_j]$,
for all $0 \leq i \leq s$.

We construct a sequence of antidirected caterpillars
$C_0\subseteq C_1\subseteq \dots \subseteq C_s$. Each $C_i$ consists of an antidirected path $P_i$, called \defi{spine}, with attached leaves.
For each $i\in [s]$, we will also find a pair of distinct vertices $x_i,y_i\in V(C_i)$ such that:
\begin{enumerate}[label = \textbf{(S\arabic*)}]
    \item $x_i$ and $y_i$ are adjacent in $C_i$,\label{S1}
    \item $x_i$ is an end point of $P_i$, and\label{S2}
    \item $y_i$ is a leaf attached to $P_i$.\label{S3}
\end{enumerate}
We call $x_i$ the \defi{current tip} and $y_i$ the \defi{current leaf}.

We also construct a sequence of homomorphisms $\psi_0,\dots,\psi_s$ such that $\psi_i:V(T_i)\longrightarrow V(C_i)$ satisfies the following for all $v\in V(C_i)$,
\begin{equation}
    \label{eq:space_in_cat}
    |\psi_i\inv(v)|\leq (1-2\eps^{1/4})m + i\Delta^{10\Delta/\beta},
\end{equation}
and for all $v\in V(C_i)\setminus \{x_i,y_i\}$, we have
\begin{equation}
    \label{eq:full_cat_i}
    |\psi_i\inv(v)|\geq (1-3\eps^{1/4})m.
\end{equation}

Furthermore, we will have $\psi_i|_{V(T_j)} = \psi_j$ for all $j<i$, $|C_i|\leq 2(i+1)$ and, if $u\in V(T_i)$ is a source (resp.\ sink), then $\psi_i(u)$ is a source (resp.\ sink) in $C_i$. We also ensure that if $a$ and $b$ are adjacent vertices in $T$, then the clusters $\psi(a)$ and $\psi(b)$ are adjacent in $R$.

We may assume that $r(T)$ is a source in $T$ (if it is a sink, we do as follows with all orientations reversed).

We start with $C_0$ consisting of an arc $x_0y_0$ (directed from $x_0$ to $y_0$), and think of $C_0$ as having spine $P_0 = \{x_0\}$ and $y_0$ as a leaf attached. 
The homomorphism~${\psi_0:V(T_0) \longrightarrow V(C_0)}$ is defined by setting~${\psi_0(\pi_0) = x_0}$.
Clearly, the caterpillar $C_0$ and $\psi_0$ satisfy~\eqref{eq:space_in_cat}~and~\eqref{eq:full_cat_i}.
Moreover, $x_0$ and $y_0$ satisfy~\ref{S1}, \ref{S2} and~\ref{S3}.

Now, suppose that we have already constructed $C_i$ and $\psi_i$ for $0\leq i\leq s-1$. Let $p$ be the parent of~$\pi_{i+1}$ in~$T$.
First, assume that $\pi_{i+1}$ is a seed, i.e.~$\pi_{i+1} \in S$. We let $C_{i+1} := C_i$, $P_{i+1} := P_i$, $x_{i+1} := x_i$,
and $y_{i+1} := y_i$. We extend $\psi_i$ by
letting $\psi_{i+1}(\pi_{i+1})$ be an arbitrary neighbour of $\psi_i(p)$ in $C_i$. Note that~\eqref{eq:space_in_cat}~and~\eqref{eq:full_cat_i}
hold for $C_{i+1}$ and $\psi_{i+1}$. 
Thus we can assume that $\pi_{i+1}$ is not a seed, i.e.~$\pi_{i+1} \in \cT$.

We now construct a caterpillar $C_{i+1}$ from $C_i$ depending on how many vertices have already been embedded to each of $x_i$ and $y_i$. We assume that the antidirected property of the caterpillar is maintained, i.e.,\ in what follows, if a new vertex is added that is adjacent to a sink (resp.\ source), then the arc is oriented towards the sink (resp.\ source). In the following, we say that a vertex $w$ of $C_i$ is
\defi{full} if $|\psi_i^{-1}(w)| > (1-3\eps^{1/4})m$.

\begin{enumerate}[label = \textbf{Case \arabic*:}, leftmargin = \widthof{Case0000}]
    \item The current tip and leaf are not full.   

       In this case, the new caterpillar is the same as before, i.e.~we let $x_{i+1} := x_{i}$, $y_{i+1} := y_i$, $C_{i+1} := C_i$ and $P_{i+1} := P_i$.  

        \vspace{0.3cm}
        
    \item Only the current leaf is full.

        We add a new leaf to $C_i$ and update the current leaf. That is, we
        define $C_{i+1}$ by adding a new vertex $a$ that is adjacent to $x_i$, and let $x_{i+1} := x_i$ and $y_{i+1} := a$.
        Since we only added a new leaf to $P_i$ in order to obtain $C_{i+1}$,
        the spine of the caterpillar is unchanged, i.e. $P_{i+1} := P_i$.
    
        \vspace{0.3cm}
    
    \item Only the current tip is full.

        We use the current leaf to extend the spine.
        Thus, define $C_{i+1}$ by adding a new vertex~$a$ that is adjacent to~$y_i$, and let~$x_{i+1} := y_i$ and $y_{i+1} := a$.        
        The new caterpillar $C_{i+1}$ has spine $P_{i+1}$ obtained from $P_i$ by adding $y_i$ and the corresponding arc between $x_i$ and $y_i$.

        \vspace{0.3cm}
    
    \item The current tip and leaf are full.

        Here, we have to create a new current tip and a new leaf. Hence,
        define $C_{i+1}$ by adding a new vertex $a$ that is adjacent to $y_i$ and an additional vertex~$b$ that is adjacent to $a$, and let $x_{i+1} := a$ and $y_{i+1} := b$. The new caterpillar $C_{i+1}$ has spine $P_{i+1}$ obtained from
        $P_i$ by appending $y_i$ and then $a$ to the tip $x_i$ of $P_i$.
\end{enumerate}

Note that we obtain $C_{i+1}$ either by attaching a new leaf to $P_i$ (Case 2), by
extending $P_i$ so that a leaf adjacent to the current tip becomes the new tip (Case 3), or by performing
both operations (Case~4). By~\ref{S1}, \ref{S2} and~\ref{S3}, in each case, $C_{i+1}$ consists of a spine $P_{i+1}$
with leaves attached to it, so $C_{i+1}$ is an antidirected caterpillar. Also, 
note that the current tip and leaf satisfy~\ref{S1}, \ref{S2} and~\ref{S3}.
Furthermore, in every case neither $x_{i+1}$ nor $y_{i+1}$ are full, while
every $v \in V(C_{i+1})\setminus \{x_{i+1},y_{i+1}\}$ is full.
Since $\psi_{i+1}$ will be an extension of $\psi_i$, it will be clear that~\eqref{eq:full_cat_i} holds for $i+1$. Finally, since $C_{i+1}$ contains at most two more vertices than $C_i$, we have $|C_{i+1}|\leq 2(i+2)$.

Now we will extend $\psi_{i}$ to get $\psi_{i+1}$. 
Let $Q$ be the shortest path from $\psi_{i+1}(p)$ to $x_{i+1}$ in $C_{i+1}$ and write $Q = z_0z_1\dots z_q$ so that $z_0 = \psi_{i+1}(p)$ and $z_q = x_{i+1}$. For $1\leq j\leq q-1$, let $L^j$ denote the set of vertices in $\pi_{i+1}$ at distance $j$ from $p$ in $T$. 

We assign all vertices in $L^j$ to the vertex $z_j$ in $Q$. If $p$ is a source in $T$ then all vertices at even (resp.\ odd) distance to $p$ in $T$ are sources (resp.\ sinks), and since these vertices of $\pi_{i+1}$ will be mapped to $z_j$ for some even $j$ (resp.\ odd $j$) and $Q$ is an antidirected path, then clearly $\psi_{i+1}(u)$ is a source if and only if $u$ is a source, for all $u\in \bigcup_{j \leq q-1} L^j$. The same holds if $p$ is a sink.

For all remaining vertices $u\in V(\pi_{i+1}) \setminus \bigcup_{j\leq q-1} L^j$, if $u$ is a source, then $u$ is assigned to whichever of $x_{i+1}$ or $y_{i+1}$ is a source in $C_{i+1}$, and if $u$ is a sink, then it is assigned to the one that is a sink in $C_{i+1}$. It is clear that arcs are preserved under this assignment and thus $\psi_{i+1}$ is a homomorphism of $T_{i+1}$.

Note that, for $v \in \{x_{i+1}, y_{i+1}\}$, we have
\[
    |\psi^{-1}_{i+1}(v)| \leq |\psi^{-1}_{i}(v)| + |\pi_{i+1}| \leq (1-3\eps^{1/4})m + \beta k \leq (1-2\eps^{1/4})m,
\]
which satisfies~\eqref{eq:space_in_cat}. Now, let $z \in V(C)\setminus \{x_{i+1},y_{i+1}\}$.
There are at most $\max_{j\leq {q-1}}|L^{j}|\leq \Delta^{|Q|}\leq \Delta^{2(i+2)}\leq \Delta^{10\Delta/\beta}$ vertices in $V(T_{i+1}) \setminus V(T_i)$ that get mapped to $z$ via $\psi_{i+1}$. So,
\[
    |\psi^{-1}_{i+1}(z)| \leq |\psi^{-1}_{i}(z)| + \Delta^{10\Delta/\beta} \leq (1-2\eps^{1/4})m + (i+1)\Delta^{10\Delta/\beta},
\]
which satisfies~\eqref{eq:space_in_cat}. Therefore, $\psi_{i+1}$ and $C_{i+1}$ satisfy all of the desired properties. At the end of this process, we let $C := C_s$ and $\psi := \psi_s$. Note that, since $s \leq 4\Delta/\beta$, we
have that
\[
    |\psi\inv(v)| \leq (1-2\eps^{1/4})m + s\Delta^{10\Delta/\beta}<(1-\eps^{1/4})m.
\]

Finally, by~\eqref{eq:size_of_cat}, we apply \cref{prop:ST_caterpillar} to find an embedding $\tau$ of $C$ in $R$ and set $\phi = (\tau \circ \psi)$.
A moment of thought reveals that $\phi$ satisfies~\eqref{eq:assignment}.

\textbf{Step 4 (Embedding of $\boldsymbol{T}$)}.
This step follows the usual embedding of trees using the regularity lemma (see, for instance,~\cite{Stein_Zarate-Gueren_2024} and~\cite{BPS1}). Below, we give a detailed sketch of the proof.
We embed $T$ following the
ordering $(\pi_0,\dots,\pi_s)$ of~${S \cup \cT}$. Recall that $\pi_0 = r(T)$
and $T_i = T[\cup_{j \leq i} \pi_j]$ is connected, for all $i \leq s$. 
For convenience, in this step, we treat each seed as a subtree of size one, which allows us
to write $V(\pi_\ell)$, for a seed $\pi_\ell \in S$. Moreover, given
an embedding $\sigma: V(H)\rightarrow V(D)$ of a digraph $H$ into $D$ and a set $X \subseteq V(D)$, we write
$\Free_{\sigma}(X) := X \setminus \sigma(V(H))$ for the set of vertices in $X$ that are
not used by $\sigma$.

For every $i \in [t]$, partition $V_i$ arbitrarily into two sets $V^{S}_i$ and $V^{P}_i$
with $|V_i^S| = \sqrt{\eps}m$. The sets~$V_i^{S}$ and~$V_i^{P}$ will be called the \defi{$S$-slice} and
the \defi{$P$-slice} of $V_i$, respectively. The $S$-slice of each cluster is used exclusively
to embed seeds, and the $P$-slice to embed vertices of the pieces of $\cT$. Note that the seeds
fit into the $S$-slices, since $|S| \leq 2/\beta \ll \sqrt{\eps}m$.
Moreover, by Fact~\ref{fact:typical} (b), we get, for every $\diamond \in \{+,-\}$ and 
every $V_j \in N^{\diamond}_R(V_i)$, that every pair $(X,Y)$ with $X \in \{V_i^{S},V_i^{P}\}$
and~$Y \in \{V_j^{S}, V_j^{P}\}$ is $(\sqrt{\eps}, \diamond,d/2)$-regular.

For every $0 \leq \ell \leq s$, we will show that there exists an embedding $\sigma_\ell: V(T_\ell) \rightarrow V(D)$ of
$T_\ell$ into $D$ such that the following holds for every $u \in V(T_\ell)$.
\begin{enumerate}[label = \textbf{(E\arabic*)}]
    \item if $u$ is a seed, then $\sigma_{\ell}(u)$ is in the $S$-slice of $\varphi(u)$; otherwise $\sigma_{\ell}(u)$ is in the $P$-slice of $\varphi(u)$. \label{E1}
    \item if $u$ is a $\diamond$-vertex for some $\diamond\in \{+,-\}$, then 
    \[
     \sigma_\ell(u) \textnormal{ is } \diamond\textnormal{-typical to } \Free_{\sigma_{\ell-1}}(Y^{S})  \textnormal{ and to } \Free_{\sigma_{\ell-1}}(Y^{P})\textnormal{ for all } Y\in \phi(N^{\diamond}_T(u)).
     \]\label{E2}
\end{enumerate}

Note that the $P$-slice of each cluster has $(1-\sqrt{\eps})m$ vertices, while by~\eqref{eq:assignment},
the map $\varphi$ assigns at most $(1-\eps^{1/4})m$ vertices of $T$ to each cluster. Hence
every embedding $\sigma_\ell$ satisfying~\ref{E1} has
\begin{equation}
    \label{eq:free}
    |\Free_{\sigma_{\ell}}(V_i^{P})| \geq \left(1-\sqrt{\eps}\right)m - (1-\eps^{1/4})m \gg \frac{\eps^{1/4}}{2}m, \text{ for every } i \in [t].
\end{equation}
Moreover, since $|S| \ll \sqrt{\eps}m$, we have
\begin{equation}
    |\Free_{\sigma_{\ell}}(V_i^{S})| \gg \frac{\sqrt{\eps}}{2}m, \text{ for every } i \in [t].
\end{equation}
We start by finding an embedding $\sigma_0$ of $T_0$, which consists of a single vertex $\pi_0 = r(T)$.
Let $X := \varphi(\pi_0)$ be the cluster to which $\pi_0$ is assigned,
and let $\diamond \in \{+,-\}$ be such that $\pi_0$ is an $\diamond$-vertex. Since $\pi_0$ is a seed,
property~\ref{E1} requires us to embed it in the $S$-slice of $X$.

To satisfy property~\ref{E2}, observe the following. Let $v \in N_T^{\diamond}(\pi_0)$ and put $Y := \phi(v)$. Then $(X^{S},Y^{S})$ and $(X^{S}, Y^{P})$ are $(\sqrt{\eps},\diamond,d/2)$-regular pairs and by~\cref{fact:typical}~(a), all but~$2\sqrt{\eps}|X^{S}|$ vertices of $X^{S}$ are $\diamond$-typical to $Y^{S}$ and $Y^{P}$. Since there are at most $\Delta$ such neighbours $v$,  all but at most~$2\sqrt{\eps} \Delta |X^{S}|$ vertices of $X^{S}$ are $\diamond$-typical to all $Y^{S}$ and $Y^{P}$ where $Y\in \phi(N^{\diamond}_T(\pi_0))$. Choose any such vertex~${x\in X^{S}}$ and set $\sigma_0(\pi_0) = x$, so that~\ref{E1} and~\ref{E2} are satisfied.

Suppose now that we wish to embed $\pi_\ell$ for some $\ell\geq 1$, and let $\sigma_{\ell-1}$ be an embedding of
$T_{\ell-1}$ satisfying~\ref{E1} and~\ref{E2}. Let $p$ be the parent of $\pi_\ell$ in $T$ and
write $q := |V(\pi_\ell)|$. Since $p \in V(T_{\ell-1})$, it is already embedded by $\sigma_{\ell-1}$.
Let $(u_0,\dots,u_{q})$ be an ordering of $V(\pi_\ell) \cup \{p \}$
with $u_0 = p$ such that, for every $i\in [q]$, exactly one $j<i$ satisfies $u_j \in N_T(u_i)$.
We construct a sequence $\sigma_{\ell}^{0},\dots,\sigma_{\ell}^{q}$ of extensions of $\sigma_{\ell-1}$,
where $\sigma_{\ell}^{0} = \sigma_{\ell-1}$ and each $\sigma_{\ell}^{i}$ extends $\sigma_{\ell}^{i-1}$ by embedding $u_i$. 

These will
satisfy analogous versions of properties~\ref{E1} and~\ref{E2}. That is, $u_0$ is already embedded into an $S$-slice, whereas all $u_i$ with $i\geq 1$ will be embedded into a $P$-slice, and moreover, if $u_i$ is a $\diamond$-vertex, then $\sigma_{\ell}^i(u_i)$ is $\diamond$-typical to $\Free_{\sigma_{\ell}^{i-1}}(Y^P)$ for all $Y\in \varphi(N^\diamond_T(u_i))$.
Setting $\sigma_{\ell} = \sigma_{\ell}^{q}$ then gives an embedding of $T_{\ell}$ into $D$.

Thus, assume we are about to embed a vertex $u_i \in V(\pi_\ell)$, for some $i \geq 1$, and let $u_j$ be its parent, that we have already embedded to satisfy the desired properties.
Let~$\diamond\in \{+,-\}$ be such that $u_i$ is a $\diamond$-vertex,
and let $\ast\in \{+,-\} \setminus \{\diamond\}$.
Moreover, let $X$ be the cluster mapped by $\varphi$ to embed $u_i$.
By~\eqref{eq:free}, we get
\[
  |N^{\ast}(u_j) \cap \Free_{\sigma_{\ell}^{i-1}}(X^{P})| \geq |N^{\ast}(u_j)\cap \Free_{\sigma_{\ell-1}}(X^{P})| - (i-j) \geq (d/2-\sqrt{\eps})\frac{\eps^{1/4}}{2}m - \beta k \gg \sqrt{\eps}|X^{P}|.
\]

Moreover, using that $|S| \leq 2/\beta$, we see that  the number of vertices in $N^{\ast}(u_j) \cap \Free_{\sigma_{\ell}^{i-1}}(X^{S})$
is also at least~$\sqrt{\eps}|X^{S}|$. Consider any~$Y\in \phi(N_T^{\diamond}(u_i))$. By~\eqref{eq:free},
the number of unused vertices in $Y^{S}$ and $Y^{P}$ are greater than~$\sqrt{\eps}|Y^S|$ and $\sqrt{\eps}|Y^{P}|$, respectively.
In particular, all these sets are $\sqrt{\eps}$-significant. Write
\[
    \mathcal{Y} := \{ \Free_{\sigma_{\ell}}(Y^S),\; \Free_{\sigma_{\ell}}(Y^{P}) : Y \in \phi(N^{\diamond}_T(u_i))\},
\]
and say a vertex $x \in V(D)$ is \defi{good} if it is $\diamond$-typical to all sets in $\mathcal{Y}$.

As before, by~\cref{fact:typical}~(a), all but at most $2\Delta\sqrt{\eps}\,|X^{S}|$
vertices of~$X^{S}$, and all but at most $2\Delta\sqrt{\eps}\,|X^{P}|$ vertices
of~$X^{P}$, are $\diamond$-typical to every set in~$\mathcal{Y}$.

Finally, if $u_i \in S$, choose a good vertex $x\in \Free_{\sigma_{\ell}}(N^{\ast}(u_j) \cap X^S)$;
otherwise choose a good vertex $x \in \Free_{\sigma_{\ell}}(N^{\ast}(u_j)\cap X^{P})$. Define $\sigma(u_i) = x$.
This implies that~\ref{E1} and~\ref{E2} are satisfied for $\sigma_{\ell}^{i}(u_i)$.
\section{Proof of \cref{thm:balance_trees}}

The proof of \cref{thm:balance_trees} follows the same lines as that of \cref{thm:main_dense},
with three differences: Step 2 is omitted, the caterpillar in Step 3 is constructed differently, and
the embedding of $T$ in Step 4 is simpler. Step 1 is unchanged.

\textbf{Step 0 (Setting up)}.
Given $\alpha > 0$ we introduce new constants $k_0, \gamma, \eps$ and $d$ such that
\[  
1/k_0 \ll \gamma \ll \eps \ll d \ll \alpha.
\]

Let $n \geq k> k_0$, and let $D$ be a digraph on at least $(1+\alpha)kn$ arcs
where $k\geq \alpha n$, and let~$T$ be an antidirected tree on $k$ arcs, with root~$r$ and at most $\gamma k$ vertices at distance $i$ from $r$ in $T$, for all $i\geq 0$. We denote by~$L_i$ the set of vertices in $T$ at a distance exactly $i$ from $r$. So, we have~${|L_i|\leq \gamma k}$ for all $i\geq 0$. 

\textbf{Step 3 (Finding a good assignment)}. We will find a homomorphism $\phi:V(T) \longrightarrow V(R)$ such that
\begin{equation}\label{eq:assignment_layers}
    |\phi\inv(V_i)| < (1- \sqrt{\eps})|V_i|\;\text{ for all } i \in [t],
\end{equation}
and moreover, for each $\ell\geq 0$, all elements of $L_\ell$ are assigned to a single cluster of $R$.

Again, it will be sufficient to find an antidirected caterpillar $C$ and a homomorphism $\psi:V(T)\longrightarrow V(C)$ satisfying the following property for
all $v \in V(C)$, except for two vertices,
\begin{equation}
    \label{eq:full_cat_layers}
    (1-\alpha^2/8)m \leq |\psi\inv(v)|\leq (1-\sqrt{\eps})m,
\end{equation}
where additionally we impose that for each $\ell\geq 0$, all elements of $L_\ell$ are assigned to the same vertex of~$C$.
In an identical way to the proof of \cref{thm:main_dense}, we will have $|C|\leq (1+\alpha^2)\kappa$.

Let $h_0$ be the maximum distance of a vertex in $T$ from $r$, and let $h = \ceil{h_0/2}$.
We construct $C$ and $\psi$ in $h$ steps. We define $T_i := T[\cup_{j \leq i} L_i]$ for all $0 \leq i \leq h_0$, and note that this is always connected and $T = T_{h_0}$. If $h_0$ is odd, then note that $L_{2h} = L_{h_0+1} = \emptyset$.

Similarly to \textbf{Step 3} in~\cref{thm:main_dense}, we construct a sequence of caterpillars
$C_0\subseteq C_1\subseteq \dots \subseteq C_h$. Each~$C_i$ consists of a spine $P_i$ with attached leaves.
For each $i\in [h]$, we will also find a pair of vertices $x_i,y_i\in V(C_i)$
satisfying ~\ref{S1}, \ref{S2} and~\ref{S3}. In contrast to~\cref{thm:main_dense}, here we 
must also keep track of whether the current tip is a sink or a source of the spine. Thus, for each $0\leq i \leq h$, 
we fix $\circ_i \in \{+,-\}$ such that $x_i$ is an $\circ_i$-vertex of $P_i$.

We also construct a sequence of homomorphisms $\psi_0,\dots,\psi_h$ such that $\psi_i:V(T_{2i})\longrightarrow V(C_i)$ satisfies the following for all $v\in V(C_i)$,
\begin{equation}
    \label{eq:space_in_cat_layers}
    |\psi_i\inv(v)|\leq (1-\sqrt{\eps})m,
\end{equation}
and for all $v\in V(C_i)\setminus \{x_i,y_i\}$, we have
\begin{equation}
    \label{eq:full_cat_i_layers}
    |\psi_i\inv(v)|\geq (1-2\sqrt{\eps})m.
\end{equation}
Moreover, if $r$ is a $\diamond$-vertex for $\diamond\in \{+,-\}$, we have 
\begin{equation}
    \label{eq:alt_spine}
    \psi_i(L_{2i-1})= \begin{cases}
        y_i, \text{ if } \circ_i = \diamond,\\
        x_i, \text{ otherwise}.
    \end{cases} \quad \text{and} \quad \psi_i(L_{2i})= \begin{cases}
        x_i, \text{ if } \circ_i = \diamond,\\
        y_i, \text{ otherwise.}
    \end{cases}
\end{equation}
Furthermore, we will have $\psi_i|_{V(T_j)} = \psi_j$ for all $j<2i$, and, if $u\in V(T_{2i})$ is a source (resp.\ sink), then $\psi_i(u)$ is a source (resp.\ sink) in $C_i$. 

We may assume that $r$ is a source in $T$ (if it is a sink, we do as follows with all orientations reversed).
Note that this implies all vertices belonging to $L_{2i}$ for some $i
\geq 0$ are sources, and all vertices belonging to $L_{2i+1}$ for some $i\geq 0$ are sinks.

We start with $C_0$ consisting of an arc $x_0y_0$ (directed from $x_0$ to $y_0$), and think of $C_0$ as having spine $P_0 = \{x_0\}$.
Moreover, we set $\circ_0 = +$.
The homomorphism~${\psi_0:V(T_0) \longrightarrow V(C_0)}$ is defined by setting~${\psi_0(r) = x_0}$.
Clearly, the caterpillar $C_0$ and $\psi_0$ satisfy~\eqref{eq:space_in_cat_layers}~and~\eqref{eq:full_cat_i_layers}. 
Furthermore, $x_0$ and $y_0$ satisfy~\ref{S1}, \ref{S2} and~\ref{S3}.

Now, suppose that we have already constructed $C_i$ and $\psi_i$ for $0\leq i\leq h-1$.

As before, we construct a caterpillar $C_{i+1}$ from $C_i$ depending on how many vertices have already been embedded to each of $x_i$ and $y_i$. 
We follow Cases 1--4 from \textbf{Step 3} in the proof of~\cref{thm:main_dense}. Here, however, a vertex $w$ of $C_i$ is
\defi{full} if $|\psi_i^{-1}(w)| > (1-2\sqrt{\eps})m$. Moreover, we set $\circ_{i+1} = \circ_i$ in Cases 1,2 and~4, and
$\circ_{i+1} = \{+,-\} \setminus \circ_i$ in Case 3.

Now we will extend $\psi_{i}$ to get $\psi_{i+1}$. If the current tip is a source, i.e. $\circ_{i+1} = +$,
then we assign $\psi(u) = y_{i+1}$ for every $u\in L_{2i+1}$ and $\psi(u) = x_{i+1}$ for every~${u\in L_{2i+2}}$; otherwise
we assign $\psi(u) = x_{i+1}$ for every $u \in L_{2i+1}$ and $\psi(u) = y_{i+1}$ for every $u\in L_{2i+2}$.
It is clear that arcs are preserved under this assignment and thus $\psi_{i+1}$ is a homomorphism of $T_{2i+2}$ satisfying~\eqref{eq:alt_spine}.

Note that, for $v \in \{x_{i+1}, y_{i+1}\}$, since $L_{2i+1}$ and $L_{2i+2}$ both have size at most $\gamma k$, we have
\[
    |\psi^{-1}_{i+1}(v)| \leq |\psi^{-1}_{i}(v)| + \gamma k \leq (1-\sqrt{\eps})m,
\]
which satisfies~\eqref{eq:space_in_cat_layers}. We also know that every $z \in V(C_{i+1})\setminus \{x_{i+1},y_{i+1}\}$ satisfies
\[
    |\psi^{-1}_{i+1}(z)| = |\psi^{-1}_{i}(z)| \leq (1-\sqrt{\eps})m,
\]
which satisfies~\eqref{eq:space_in_cat_layers}. Therefore, $\psi_{i+1}$ and $C_{i+1}$ satisfy all of the desired properties. At the end of this process, we let $C := C_h$ and $\psi := \psi_h$. 

In the same way as in \cref{thm:main_dense}, $|C|\leq (1+\alpha^2)k$ and applying \cref{prop:ST_caterpillar} yields an embedding $\tau$ of $C$ in $R$
and we set $\phi = (\tau \circ \psi)$, so that~\eqref{eq:assignment_layers} is satisfied.

\textbf{Step 4 (Embedding of $\boldsymbol{T}$)}.
Enumerate the vertices of $V(T)$ by $u_0,\dots,u_{k}$ so that $u_0 = r$ and all vertices in $L_j$ come earlier than those in $L_{j'}$, for all $j<j'$. In particular, for all $i\geq 1$, each $u_i$ has exactly one neighbour $u_{i'}$ satisfying $i'<i$. For $0\leq i \leq k$ we define $T^i \coloneqq T[\cup_{j\leq i}u_j]$. We find a copy $\sigma$ of the tree $T$ in $D$ by embedding the elements one-by-one using this ordering, starting with $r$, maintaining that
\begin{equation}\label{eqn:embedding_via_assignment_layers}
    \sigma(u) \in \phi(u) \textnormal{ for all } u \in V(T).
\end{equation}
For any $u\in V(T)$, note that all neighbours $v$ of $u$ succeeding $u$ in the ordering, have the same distance to $r$. Thus we can define $Y_u$ to be the unique cluster that contains $\phi(v)$ for all neighbours $v$ of $u$ that succeed $u$. We will show that, if $u_i\in V(T)$ is a $\diamond$-vertex for some $\diamond\in \{+,-\}$, then 
\begin{equation}\label{eqn:typical_property_layers}
     \sigma(u_i) \textnormal{ is } \diamond\textnormal{-typical to } Y_{u_i}\setminus \sigma(T^{i-1}).
\end{equation}
We first embed $u_0=r$. Let $X = \phi(u_0)$ and let $\diamond \in \{+,-\}$ be such that $u_0$ is an $\diamond$-vertex. Now, $(X,Y_{u_0})$ is an $(\eps,\diamond,d)$-regular pair and by~\cref{fact:typical}~(a), all but $\eps m$ vertices of $X$ are $\diamond$-typical to $Y_{u_0}$. Choose one such vertex $x\in X$ for the image of $u_0$, so that~\eqref{eqn:typical_property_layers} is satisfied.

Suppose now we wish to embed $u_i$ for $i\geq 1$, and let $u_j$ be the neighbour of $u_i$ which precedes it. Let~$\diamond\in \{+,-\}$ be such that $u_i$ is a~$\diamond$-vertex. Let~$\ast\in \{+,-\} \setminus \{\diamond\}$. Let $X = \phi(u_i)$, and note that by~\eqref{eq:assignment_layers} and~\eqref{eqn:embedding_via_assignment_layers},  sets $X\setminus \sigma(T^{i-1})$ and $Y_{u_i} \setminus \sigma(T^{i-1})$ both have size greater than $\sqrt{\eps}m$. In particular, $Y_{u_i}\setminus \sigma(T^{i-1})$ is an $\eps$-significant set. As before, $(X,Y_{u_i})$ is an $(\eps,\diamond,d)$-regular pair.
By definition of $Y_{u_j}$, we have $X  = Y_{u_j}$.
Since~\eqref{eqn:typical_property_layers} is satisfied for $u_j$, it follows that $\sigma(u_j)$ is $\ast$-typical to $X\setminus \sigma(T^{i-1})$, and therefore the set $X'$ of $\ast$-neighbours of $\sigma(u_j)$ in $X\setminus \sigma(T^{i-1})$ has size at least $(d-\eps)|X\setminus \sigma(T^{i-1})| > (d-\eps)\sqrt{\eps}m>\eps m$.
Similarly as before, by~\cref{fact:typical}, all but at most $\eps m$ vertices of $X$ are $\diamond$-typical to $Y_{u_i}\setminus  \sigma(T^{i-1})$. Choose one such vertex $x\in X'$ that satisfies this property, and define $\sigma(u_i) = x$. This implies that~\eqref{eqn:typical_property_layers} is satisfied for $\sigma(u_i)$, and clearly~\eqref{eqn:embedding_via_assignment_layers} holds. This completes our argument and we find a copy $\sigma$ of $T$ in $D$, as desired.
\section{Applications}\label{sec:applications}

In this section we discuss several consequences of \cref{thm:main_dense}. In particular, we obtain an asymptotic version of Burr's conjecture for bounded degree antidirected trees, and linear upper bounds on the corresponding directed and oriented Ramsey numbers.

\subsection{Burr's conjecture}

A celebrated result in extremal graph theory, known as the Gallai–Hasse–Roy–Vitaver (GHRV) theorem~\cite{gallai1968directed,hasse1965algebraischen,roy1967nombre,vitaver1962determination}, proved independently by these four authors, considers the chromatic number of a digraph that forces the existence of a directed path. The chromatic number of a digraph~$D$ is the smallest integer~$m$ for which there exists a colouring of $V(D)$ with~$m$ colours such that no two adjacent vertices receive the same colour. The GHRV theorem shows that every orientation of a graph with chromatic number $k+1$ contains a copy of the directed path with $k$ arcs. To embed more general oriented trees, one needs to increase the chromatic number threshold, since any regular tournament on $2k-1$ vertices (and thus with chromatic number $2k-1$) does not contain
an out- or in-star with $k$ arcs. 

In 1980, Burr~\cite{burr1980subtrees} conjectured the following optimal bound.
\begin{conjecture}[Burr \cite{burr1980subtrees}]\label{conj:Burr}
    Every digraph with chromatic number at least $2k$ contains a copy of every oriented tree $T$ with $k$ arcs.
\end{conjecture}

Burr proved in 1980 that chromatic number at least $k^2$ is sufficient, and this was not improved until more than thirty years later when Addario-Berry, Havet, Linhares Sales, Reed and Thomassé~\cite{Addarioberry2013oriented} showed that $(k^2+k+3)/2$ is sufficient\footnote{In~\cite{Addarioberry2013oriented}, the results are stated for trees with $k$ vertices and thus the bounds read slightly differently.}. The first sub-quadratic bound was obtained by Bessy, Gonçalves and Reinald~\cite{bessy2024oriented}, who proved that chromatic number $Ck^{3/2}$ suffices for some absolute constant $C$.
For antidirected trees, considerably stronger bounds are known. Addario-Berry, Havet, Linhares Sales, Reed and Thomassé~\cite{Addarioberry2013oriented} proved that every digraph with chromatic number at least $5k-7$ contains every antidirected tree with $k$ arcs. As an immediate consequence of \cref{thm:main_dense}, we lower this to~${(2+o(1))k}$ for antidirected trees of bounded maximum degree, therefore showing that Burr's conjecture holds asymptotically for this class of trees.

\begin{corollary}
    For all $\alpha>0$ and $\Delta\in \NN$ there exists $k_0$ such that for all $n\geq k>k_0$ with $k\geq \alpha n$ the following holds. Every digraph $D$ on $n$ vertices with chromatic number at least $(2+\alpha)k$ contains a copy of every antidirected tree $T$ with $k$ arcs and $\Delta(T)\leq \Delta$.
\end{corollary}

\begin{proof}
    Let $\alpha>0$ and $\Delta\in \NN$, and choose $k_0$ to be the output of \cref{thm:main_dense} when applied with $\alpha/2$ and~$\Delta$. Let $t = (2+\alpha)k$ and suppose that $D$ is a digraph that yields a counterexample. We may assume~$D$ is $t$-critical, that is, if we delete any vertex or arc from $D$ then the resulting subgraph has chromatic number less than $t$. Then $D$ is oriented, and has minimum total degree at least $t-1$. Therefore $D$ has at least $(t-1)n/2 = (1+\alpha/2)kn$ arcs. Applying \cref{thm:main_dense} with $\alpha/2$ playing the role of $\alpha$, we obtain a copy of every antidirected tree $T$ with $k$ arcs and $\Delta(T)\leq \Delta$, as desired.
\end{proof}

A corresponding corollary can be obtained from~\cref{thm:balance_trees}.

\subsection{Directed Ramsey numbers of trees}

Ramsey numbers for tournaments and digraphs provide a natural directed analogue of classical graph Ramsey theory. We consider both tournaments and complete symmetric digraphs.
The \defi{$\ell$-colour oriented Ramsey number} of an oriented graph $H$, denoted by \defi{$\overrightarrow{r}(H,\ell)$}, is the smallest $n$ such that every $\ell$-edge-coloured tournament on $n$ vertices (where every pair of vertices form one edge) contains a monochromatic copy of $H$. The \defi{$\ell$-colour directed Ramsey number}, denoted by \defi{$\overleftrightarrow{r}(H,\ell)$}, is defined the same way but with replacing any tournament with the complete symmetric digraph on $n$ vertices. This directed Ramsey number was introduced by Bermond~\cite{BERMOND1974directedramsey}, and the corresponding notion for tournaments has since become a standard object of study. We use the notation from~\cite{bucic2017directed}.

For the directed path $P_k$ with $k$ arcs, Chv\'atal~\cite{chvatal1972monochromatic} and independently Gy\'arf\'as and Lehel~\cite{gyarfas1973ramsey} proved that $\overrightarrow r(P_k,\ell)=k^\ell$. This can also be deduced as a simple corollary of the aforementioned GHRV theorem.
Motivated by this result, Yuster~\cite{yuster2017ramsey} asked for the largest possible $\ell$-colour Ramsey number among oriented trees with $k$ vertices. Buci\'c, Letzter and Sudakov~\cite{bucic2017directed} showed that every oriented tree $T$ satisfies
$\overrightarrow r(T,\ell)\le c_\ell|T|^\ell $
and
$\overleftrightarrow r(T,\ell)\le d_\ell|T|^{\ell-1},$
where $c_\ell$ and $d_\ell$ depend only on $\ell$. These bounds are best possible up to the constant factor for some trees, such as directed paths, although for specific classes of trees substantially smaller Ramsey numbers may exist.

For antidirected trees of bounded maximum degree, \cref{thm:main_dense} implies that both the directed and oriented Ramsey numbers grow only linearly with the size of the tree.

\begin{corollary}
    For all $\alpha>0$ and $\Delta,\ell \in \NN$ with $\ell\geq 2$ there exists $k_0$ such that for all $k>k_0$ the following holds. For every antidirected tree $T$ on $k$ arcs with $\Delta(T)\leq \Delta$, we have 
    \begin{equation*}
        \overleftrightarrow{r}(T,\ell)\leq (1+\alpha)\ell k \quad \text{ and } \quad \overrightarrow r(T,\ell)\le (2+\alpha)\ell k.
    \end{equation*}
\end{corollary}

\begin{proof}
    Let $\alpha>0$, $\Delta\in \NN$ and $\ell \geq 2$.
    We may assume that $\alpha\in (0,1)$ is sufficiently small so that $\alpha\leq 1/(4\ell)$ holds, as this only strengthens the statement. Let $k_0$ be the output of~\cref{thm:main_dense} with inputs $\alpha$ and~$\Delta$.  
    We first prove $\overleftrightarrow{r}(T,\ell)\leq (1+\alpha)\ell k$.
    Let $D$ be an $\ell$-edge-coloured complete digraph on $n = (1+\alpha)\ell k$ vertices, and note that $\alpha n \leq 2\alpha \ell k \leq k$. By averaging, there is a colour $c$ that appears on at least $n(n-1)/\ell = (1+\alpha)k(n-1)> (1+\alpha/2)kn$ arcs. Applying \cref{thm:main_dense} to the spanning subdigraph consisting of all arcs of colour $c$, we find a monochromatic copy of $T$ in $D$, as desired.

    The same averaging argument, together with the observation that every tournament contains exactly half of the arcs of the complete symmetric digraph, yields the proof of $\overrightarrow r(T,\ell)\le (2+\alpha)\ell k$.
\end{proof}

The above bound on $\overleftrightarrow{r}(T,\ell)$ is asymptotically best possible up to the $(1+\alpha)$ factor. Indeed, a directed analogue of a construction of Erd\H{o}s and Graham~\cite{ErdosGraham1973} shows that, for infinitely many values of $\ell$, and for every $k$-arc tree $T$, we have $
\overleftrightarrow r(T,\ell)\ge \ell(k-1)+2$.

\section{Extension of AI proof of the Erd\H{o}s--S\'os conjecture to \cref{conj:main}}\label{sec:AI}

\begin{theorem}[Extension of~\cite{AI_pdf}]\label{thm:AI}
    Every directed graph with $e(D)> (k-1)n$ contains a copy of every antidirected tree on $k$ arcs.
\end{theorem}
We follow the proof from \cite{AI_pdf} and the exposition given by Bloom~\cite{proof_exposition}, with some small adaptations to the antidirected setting.

\begin{proof}
Let $T$ be an antidirected tree on $t\geq 2$ vertices.\footnote{We deviate from the usual notation in this paper to fix the number of vertices in the tree instead of number of arcs, to reflect the similarities to \cite{AI_pdf}.}
Let $D$ be a digraph on $n$ vertices that does not contain a copy of $T$. We will show that $e(D)\leq (t-2)n$.
\\
Let $\pi = (v_0,\dots,v_{n-1})$ be a permutation of $V(D)$, and let $1\leq j \leq n-1$ be an index. For $\diamond\in\{+,-\}$, we say that the pair $(\pi,j)$ is \defi{$\diamond$-marked} if $\diamond=-$ and  $v_jv_0$ is an arc of $D$ from $v_j$ to $v_0$ in $D$, or if $\diamond=+$ and  $v_0v_j$ is an arc of $D$  from $v_0$ to $v_j$. 
Let $M^\diamond(D)$ be the number of pairs $(\pi,j)$ that are $\diamond$-marked, for $\diamond\in\{+,-\}$. By summing over all possible vertices $v\in V(D)$ that  play the role of $v_0$ in a permutation $\pi$, we have for each $\diamond\in\{+,-\}$,
\begin{equation*}
    M^\diamond(D) = (n-1)!\sum_{v\in V(D)}\deg^\diamond(v) = (n-1)! e(D).
\end{equation*}
So we can define $M(D)=M^+(D)=M^-(D)$. 
For a root $r\in V(T)$ and $\diamond\in\{+,-\}$, let  $\mathcal C^\diamond(T,r)$ be the set of all  $\diamond$-marked pairs $(\pi,j)$ that contain a copy of $T$ in $\{v_0,\dots,v_j\}$ with $r$ mapped to $v_0$. Set $C^\diamond(T,r):=|\mathcal C^\diamond(T,r)|$. We will show that, for any $r\in V(T)$   we have
\begin{equation}\label{eqn:counting_formula}
   M(D) \leq C^{\diamond_r}(T,r) + (t-2)n!
\end{equation}
where $\diamond_r$ is such that $r$ is a $\diamond_r$-vertex. 
Note that since $T$ does not embed in $D$ by assumption, we have $C^{\diamond_r}(T,r) = 0$ for all $r\in V(T)$, and since $M(D)= (n-1)!e(D)$, proving \cref{eqn:counting_formula} is sufficient to proving \cref{conj:main}.
\\
So, fix a root $r\in V(T)$ and for simplicity, write $\diamond$ for $\diamond_r$.
 We proceed by induction on $t$. If $t = 2$, then $T$ is an edge and any $\diamond$-marked pair yields a copy of $T$ rooted at $v_0$. So  \cref{eqn:counting_formula} holds. Suppose $t\geq 3$ and the formula holds for all smaller values.

\textbf{Case I: $\boldsymbol{r}$ is a leaf in $\boldsymbol{T}$. }
Let $T'$ be the subtree obtained by deleting $r$,  let $x$ be the unique $\diamond$-neighbour of $r$ in $T$, and root $T'$ at $x$. Let $*$ be the unique element of $\{+,-\}\setminus \{\diamond\}$. By induction, $M(D)\leq C^{\ast}(T',x)+(t-3)n!$, and it remains to show that $C^\ast(T',x)\leq C^\diamond (T,r)+ n!$.

Fix a permutation $\pi = (v_0,\dots,v_{n-1})$ of $V(D)$. Consider the indices $j_1<\dots<j_{m_\pi}$ such that $(\pi,j_\ell)\in \mathcal C^\ast(T',x)$.
Note that $v_0$ is in the $\diamond$-neighbourhood of each $v_{j_\ell}$. For each $2\leq \ell\leq m_\pi$, define
\begin{equation*}
    \pi_\ell = (v_{j_\ell}, v_{j_\ell-1}, v_{j_\ell-2},\dots,v_{0}, v_{n-1},v_{n-2},\dots,v_{j_\ell +1}).
\end{equation*}
We claim that $(\pi_\ell,j_{\ell})\in\mathcal C^\diamond(T,r)$. Indeed, as $(\pi, j_\ell)$ is $\ast$-marked, we know that $(\pi_\ell,j_\ell)$ is $\diamond$-marked. 
Moreover, $(\pi,j_1)$ contains a copy of $T'$ with $x$ mapped to $v_0$, and since $j_1<j_\ell$, this does not use the vertex $v_{j_\ell}$. This can be extended to a copy of $T$ by mapping the leaf $r$ to $v_{j_\ell}$, and so, $(\pi_\ell,j_\ell)$ contains this copy of $T$ with $r$ mapped to the first element of $\pi_\ell$, as required.
Finally, observe that the mapping of $(\pi,j_\ell)$ to $(\pi_\ell,j_\ell)$ is injective and so, $C^\diamond(T,r)\ge\sum_\pi(m_\pi-1)\ge C^\ast(T',x)-n!$ which is as desired.

\textbf{Case II: $\boldsymbol{r}$ is not a leaf in $\boldsymbol{T}$. } Then $T$ is the union of two trees $T_1$ and $T_2$ that intersect only in $r$. By induction, $2M(D)\le C^\diamond(T_1, r)+C^\diamond(T_2,r)+(|T_1|+|T_2|-4)n!=C^\diamond(T_1, r)+C^\diamond(T_2,r)+(t-3)n!$. So~\eqref{eqn:counting_formula} holds if we can show that $C^\diamond ( T_1, r)+C^\diamond(T_2,r)\le C^\diamond(T,r)+n!+M(D)$, which we will do now, in the form
\begin{equation}\label{eqn:r_not_leaf}
    C^\diamond ( T_1, r) - C^\diamond(T,r)-n!\le 
M(D)
-C^\diamond(T_2,r).
\end{equation}
For this, consider a pair $(\pi, j)\in\mathcal C^\diamond(T_1,r)\setminus\mathcal C^\diamond(T,r)$. 
Then $\pi=(v_{0},\dots, v_{{n-1}})$, with the edge $v_0v_{j}$ present. Let $j'_\phi$ be the last index such that $v_{j'_\phi}$ is used by  an embedding $\phi$ of $T_1$, and let $j'$ be the minimum of $j'_\phi$ over all embeddings $\phi$ of $T_1$ into vertices $v_0,\dots,v_j$, with $r$ embedded in $v_0$. Then $1\le j'\le j$. If $j'=j$ then add $(\pi,j)$ to a set $\mathcal E$ of exceptional pairs. Assume now that $j'<j$.  As $(\pi, j)$ does not contain $T$, we know that $T_2$ is not contained in the $\diamond$-marked pair $(\pi',j-j')$ where $$\pi'=(v_{0},v_{j'+1}\dots, v_j, v_1,v_2, v_3,\dots, v_{j'-1},v_{j'},v_{j+1},v_{j+2},\dots, v_{{n-1}}).$$
In this way, we have defined a map from  $\mathcal C^\diamond(T_1,r)\setminus (\mathcal C^\diamond(T,r)\cup\mathcal E)$ to $\mathcal M(D)\setminus \mathcal C^\diamond(T_2,r)$, where $\mathcal M(D)$ is the set of all $\diamond$-marked pairs of $D$. It is not hard to see that this map is injective, as we can determine the middle part $(v_1,v_2,v_3,\dots,v_{j'})$ using the definition of $j'$, and thus recover $\pi$. So
\[
C^\diamond(T_1,r)- C^\diamond(T,r)-|\mathcal E|\le M(D) -C^\diamond(T_2,r),
\]
and this proves \eqref{eqn:r_not_leaf}, since for each permutation we added at most one pair to $\mathcal E$, and therefore, $|\mathcal E|\le n!$. This completes the proof.
\end{proof}

In the same way as in Section~\ref{sec:applications}, we can deduce the following corollaries from Theorem~\ref{thm:AI}:

\begin{corollary}
   \label{cor:burr_ai}
   Every digraph $D$ on $n$ vertices with chromatic number at least $2k$ contains a copy of every antidirected tree $T$ with $k$ arcs.
\end{corollary}

\begin{corollary}
    \label{cor:ramsey_ai}
       Let $k,\ell\geq 2$. For every antidirected tree $T$ on $k$ arcs, we have 
    \begin{equation*}
        \overleftrightarrow{r}(T,\ell)\leq \ell (k-1) +2\quad \text{ and } \quad \overrightarrow r(T,\ell)\le 2\ell(k -1) +2.
    \end{equation*}
\end{corollary}

Observe that~\cref{cor:burr_ai} settles Burr's conjecture (\cref{conj:Burr}) for antidirected trees.
Moreover, combining Corollary~\ref{cor:ramsey_ai} with the lower bound obtained by 
Erd\H{o}s and Graham~\cite{ErdosGraham1973}, yields $\overleftrightarrow{r}(T,\ell)=\ell (k-1) + 2$ for infinitely many values of $\ell$.

\section*{Acknowledgements}
This project was initiated whilst the third author was based at the Center for Mathematical Modeling as part of their PhD visiting program. She is grateful to the hosts and the Center for Mathematical Modeling for the opportunity and funding support through Proyecto Basal FB210005. 

\bibliographystyle{plain}
\bibliography{bib}

\end{document}